\documentclass{article}
\usepackage{mathtools}

\usepackage{arxiv}

\usepackage[utf8]{inputenc} % allow utf-8 input
\usepackage[T1]{fontenc}    % use 8-bit T1 fonts
\usepackage{hyperref}       % hyperlinks
\usepackage{url}            % simple URL typesetting
\usepackage{booktabs}       % professional-quality tables
\usepackage{amsfonts}       % blackboard math symbols
\usepackage{nicefrac}       % compact symbols for 1/2, etc.
\usepackage{microtype}      % microtypography
\usepackage{lipsum}
\usepackage{graphicx}
\usepackage{floatrow}
\usepackage{graphicx}
\usepackage[label font=bf, 
            labelformat=simple]{subfig}
\usepackage{caption}
\usepackage{lineno}
\usepackage{svg}
\usepackage{parskip} % For paragraph layout
\usepackage{setspace} % For using single or double spacing
\usepackage{emptypage} % To insert empty pages
\usepackage{multicol} % To write in multiple columns (executive summary)
\usepackage{epstopdf}
\PrependGraphicsExtensions{.tif, .tiff}
\usepackage{titlesec}
\titlespacing{\section}{0pt}{3.3ex}{2ex}
\titlespacing{\subsection}{0pt}{3.3ex}{1.65ex}
\titlespacing{\subsubsection}{0pt}{3.3ex}{1ex}
\usepackage{color}
\usepackage{gensymb}
\usepackage[english]{babel} % The document is in English  
\usepackage[utf8]{inputenc} % UTF8 encoding
\usepackage[T1]{fontenc} % Font encoding
\usepackage[11pt]{moresize} % Big fonts

\usepackage{graphicx}
\usepackage{transparent} % Enables transparent images
\usepackage{eso-pic} % For the background picture on the title page
\usepackage{subfig} % Numbered and caption subfigures using \subfloat.
\usepackage{tikz} % A package for high-quality hand-made figures.
\usetikzlibrary{}
\graphicspath{{./Images/}} % Directory of the images
\usepackage{caption} % Coloured captions
\usepackage{xcolor} % Coloured captions
\usepackage{amsthm,thmtools,xcolor} % Coloured "Theorem"
\usepackage{float}

\usepackage{amsmath}
\usepackage{amsthm}
\usepackage{amssymb}
\usepackage{amsfonts}
\usepackage{bm}
\usepackage[overload]{empheq} % For braced-style systems of equations.
\usepackage{fix-cm} % To override original LaTeX restrictions on sizes

\usepackage{tabularx}
\usepackage{longtable} % Tables that can span several pages
\usepackage{colortbl}

\usepackage{algorithm}
\usepackage{algorithmic}

\usepackage{pdfpages} % To include a pdf file
\usepackage{afterpage}
\usepackage{lipsum} % DUMMY PACKAGE
\usepackage{fancyhdr} % For the zheaders
\usepackage{amsthm}   % for theorem, lemma, proof environments
\usepackage{amsmath}  % you probably already have this for equations

\usepackage{comment}
\graphicspath{ {./images/} }
\newtheorem{theorem}{Theorem}          % creates Theorem 1, 2, ...
\title{Guess My Number! From Binary Tricks to General Base Representations - How Many Cards Do You Need?}

\author{
 Guglielmo Vesco\\
DFIS\\
Politecnico di Milano, MI \\
  \texttt{guglielmo.vesco@gmail.com} \\
}

\begin{document}

\maketitle

\begin{abstract}
We revisit the classic “guess my number” game and extend it from its familiar binary form to representations in any integer base. For each base we derive formulas for the number of cards needed to identify a given integer and, conversely, for the largest integer that can be determined when the number of cards is fixed. Both analysis and graphical evidence show that base~2 is optimal: it requires the fewest cards for any given number and, for a fixed card count, allows the widest range of integers to be guessed. Complete proofs are provided in the Appendix.
\end{abstract}  \keywords{binary representation, positional numeral systems, number guessing game}

% keywords can be removed

\section*{Introduction}

Walking through the streets of Naples, one might encounter Giuseppe Polone (shown in the photo in the top panel of Figure~\ref{fig:figure1}), a master of constructing magic squares, capable of producing designs as large as $25\times 25$ \cite{Polone2025}. It is said that Polone’s fascination began decades ago, when a 1978 journey through the Amazon left him stranded for weeks; to pass the time he invented crosswords and magic squares and developed novel methods for creating ever-larger ones. Among the many mathematical games he now shares with curious passersby, one is particularly striking: the guessing of a secret number between $1$ and $60$.

Polone performs this feat by displaying six special cards (see the bottom panel of Figure~\ref{fig:figure1}). The participant merely indicates on which of these cards the chosen number appears, and with only this information Polone can reveal the secret number without fail.

\begin{figure}[H]
\centering
\begin{minipage}{.8\textwidth}
\centering
 \includegraphics[width=.7\textwidth]{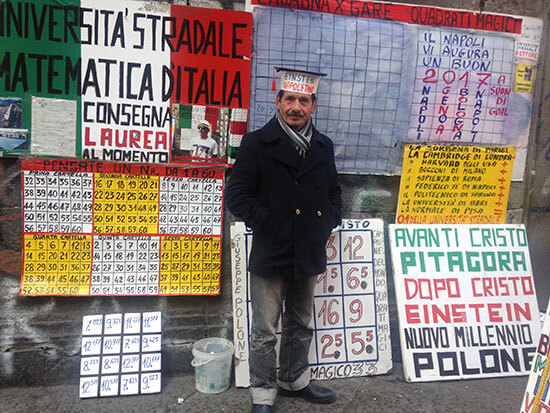}
 \vspace{10mm}
\end{minipage}

\begin{minipage}{.8\textwidth}
\centering
 \includegraphics[width=.85\textwidth]{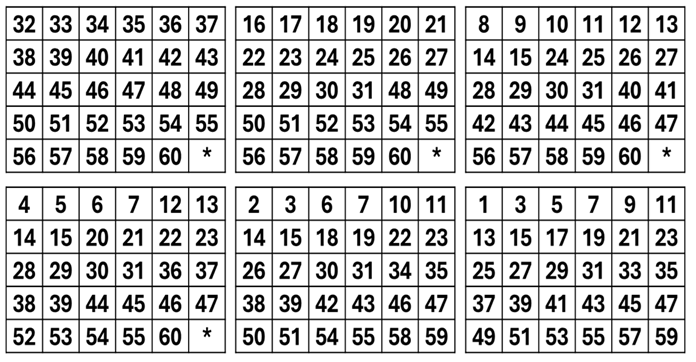}
\end{minipage}
\caption{Top: Giuseppe Polone during a street performance in Naples (photo sourced from \cite{polonephoto}).  
Bottom: the six binary tables (cards) used in the classic number-guessing game for integers from $1$ to $60$.}
\label{fig:figure1}
\end{figure}

Although captivating when performed in the street, this is in fact a well-known mathematical game (see, for example, \cite{Gardner1958,Yingprayoon2019,LongaBinaryCardTrick,ChalkdustBinaryTrick2016}), often used to introduce the concept of representing integers in base~$2$.

The idea is simple: every positive integer can be written as a sum of powers of two. Each of the six cards contains all numbers whose binary decomposition includes a specific power of two—one card for $2^5$, another for $2^4$, and so on down to $2^0=1$. By indicating the cards on which the secret number appears, the participant is effectively revealing which powers of two occur in its decomposition. The “magician” merely sums those powers to reconstruct the number.

For example, if the chosen number is $45$, it appears on the cards for $2^5=32$, $2^3=8$, $2^2=4$, and $2^0=1$. Adding these,
\[
32 + 8 + 4 + 1 = 45,
\]
the secret number is immediately determined. If $p$ is the largest exponent of two represented among the cards and $N_{card}$ represents the total number of cards, the largest number that can be guessed is
\[
2^{p+1}-1 = 2^{N_{card}} - 1 .
\]
Thus with six cards one can identify any integer from $1$ to $63$.

One may naturally ask whether the game can be adapted to other bases. Such variants, however, generally require additional information, since one must also specify the coefficient multiplying each power of the base. An elegant example is given by Sezin~\cite{Sezin2009}, who proposes a scheme based on four cards containing the integers from $1$ to $64$ written in different colors to exploit base~$4$. In this construction, each card corresponds to a power of $4$, while the color encodes the coefficient $1$, $2$, or $3$. This method does not rely on a yes/no questioning framework, but instead distinguishes multiple mutually exclusive states through color.

In the next section, we adopt a different perspective and extend the game to arbitrary bases while preserving a yes/no framework. Within this setting, we derive a general formula for the number of cards required to guess a number in any base representation. We also show that base~$2$ is optimal: it minimizes the number of cards needed for a given integer and, conversely, for a fixed number of cards, allows the largest range of integers to be guessed.

\section*{Beyond Binary: The Game in Any Base}

In a base~$b$ different from two, it is not enough to know whether a certain power of the base appears in the decomposition of a number: one must also specify the coefficient by which that power is multiplied, which can take values from $1$ up to $b-1$. For example,
\[
23 = 2\cdot 3^2 + 1\cdot 3^1 + 2\cdot 3^0
\]
is the expansion of $23$ in base~$3$.

To guess numbers using a yes/no scheme similar to Polone’s game, we therefore need a separate card for each non-zero multiple of every power of~$3$. Concretely, for each power of~$3$ we must have one card containing the numbers in which that power appears with coefficient $1$, and another card for those where it appears with coefficient $2$. In the example shown in Figure~\ref{fig:figure2}, to guess any number between $1$ and $26$ (that is, up to $3^3-1$) six cards are required: one for $3^2$, one for $2\cdot3^2$, one for $3^1$, one for $2\cdot3^1$, one for $3^0$, and one for $2\cdot3^0$.

Continuing with the example of the number $23$ discussed above, let us see how it can be guessed using the six cards shown in Figure~\ref{fig:figure2}. We observe that $23$ appears in the top central card, the top right card, and the bottom right card. These cards correspond, respectively, to the presence of $2\cdot 3^2$, $1\cdot 3^1$, and $2\cdot 3^0$ in the base-$3$ decomposition of the number. Summing these contributions, we obtain
\[
2\cdot 3^2 + 1\cdot 3^1 + 2\cdot 3^0 = 23,
\]
which correctly reconstructs the chosen number.

\begin{figure}[H]
\centering
\includegraphics[width=.7\textwidth]{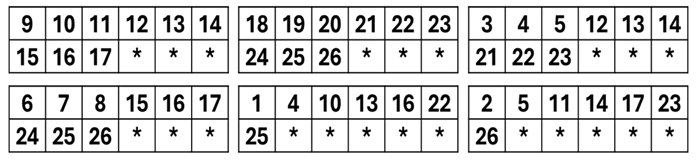}
\caption{Example of the game in base $3$: the six tables needed to guess any
number from $1$ to $26$.  
Each pair of tables corresponds to one power of $3$, with separate tables
for coefficients $1$ and $2$ of that power.}
\label{fig:figure2}
\end{figure}

In general, when numbers are represented in base~$b$, each power of~$b$ requires $(b-1)$ cards, one for each possible non-zero coefficient. If $p$ is the highest power of~$b$ needed in the decomposition of a given number $n$, the total number of cards $N_{card}$ required to guess any number in this range is
\begin{equation}
N_{card} = (b-1)\,(p+1).
\label{eq:1}
\end{equation}
In Polone's game $b=2$, and formula~\eqref{eq:1} reduces to
\[
N_{card} = (2-1)\,(p+1) = p+1 .
\]
Because $p = \lfloor \log_{b} n \rfloor$ (where $\lfloor x \rfloor$ denotes the floor of $x$), equation~\eqref{eq:1} can also be written as
\begin{equation}
N_{card} = (b-1)\,\bigl(\lfloor \log_{b} n \rfloor + 1\bigr).
\label{eq:2}
\end{equation}

Figure~\ref{fig:figure3} shows, for bases $b=2$ through $10$, how the number of cards $N_{card}$ grows with the representable integer $n$, as given by equation~\eqref{eq:2}. For each fixed base the growth is step-like: $N_{card}$ increases by $(b-1)$ each time $n$ crosses a power of $b$, reflecting the need for an additional block of $(b-1)$ cards whenever a new power enters the representation.

Figure~\ref{fig:figure4} offers the complementary view: for selected integers $n=8$ to $16$, it plots the required number of cards as a function of the base $b$. In this complementary view we can appreciate that the dependence on $b$ is not always strictly decreasing for a fixed $n$. For example, when $n=8$, both $b=2$ and $b=3$ require the same number of cards, and when $n=16$, base $4$ requires more cards than base $5$. Despite such local ties or reversals, base~$2$ always provides at least one of the minima of $N_{card}$ for every integer $n$. This statement is proved in Theorem~1 of the Appendix.

Figure~\ref{fig:figure5} displays the maximum representable number as $N_{card}$ varies, again for bases $2$ through $10$. These plots show that, for a fixed number of cards, the binary representation covers the greatest range of integers. The demonstration of this statement is provided in Theorem~2 of the Appendix.

\begin{figure}[H]
\centering
\includegraphics[width=1\textwidth]{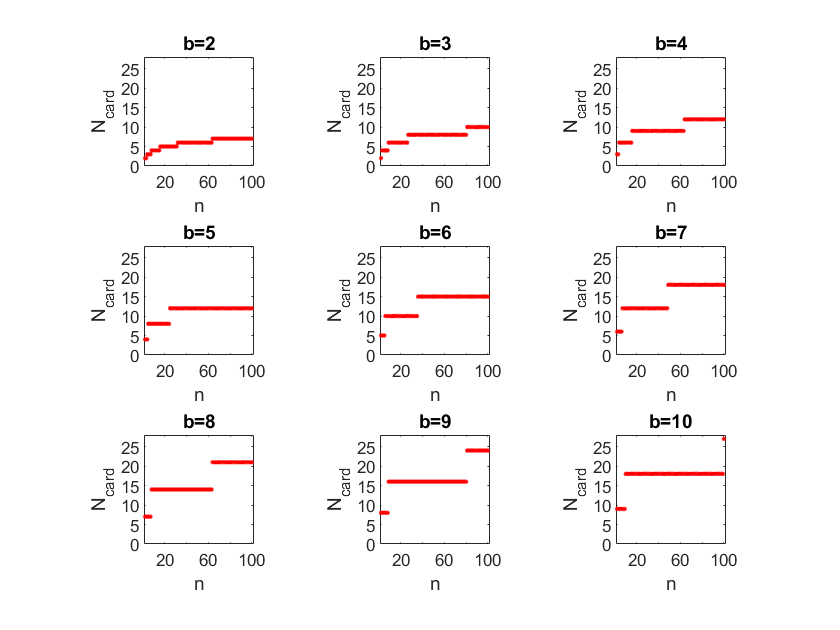}
\caption{Number of cards $N_{card}$ required to represent integers as a function of $n$, for bases $b=2,\dots ,10$.  
For each fixed base the stepwise growth reflects the addition of a new block of $(b-1)$ cards whenever $n$ passes a power of $b$.}
\label{fig:figure3}
\end{figure}

\begin{figure}[H]
\centering
\includegraphics[width=1\textwidth]{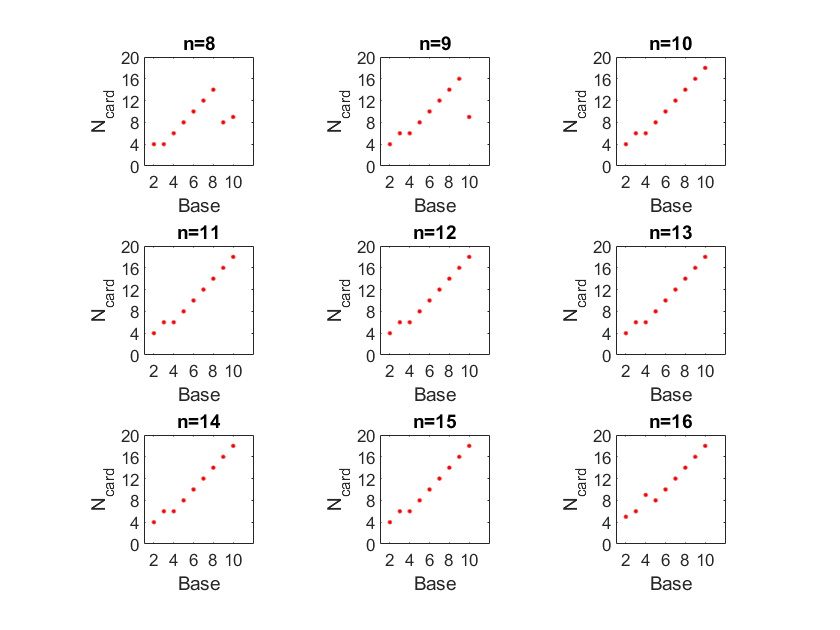}
\caption{For integers $n=8$ to $16$, the number of cards $N_{card}$ required
to guess $n$ as a function of the base $b$.}
\label{fig:figure4}
\end{figure}

Finally, it is natural to fix the number of available cards and ask which base allows the widest range of integers to be guessed. Because each power of~$b$ requires $(b-1)$ cards, the set of integers that can be represented with $N_{card}$ cards extends from $1$ up to
\begin{equation}
b^{\lfloor N_{card}/(b-1) \rfloor} - 1 ,
\end{equation}
where $p$ is the largest integer power of~$b$ supported by the available cards. The floor is needed because $p+1 = N_{card}/(b-1)$ must be an integer, and when this quotient is not itself an integer we take its integer part.\\
As a concrete example, consider $N_{\mathrm{card}}=5$ cards and base $b=3$. Since each power of $3$ requires $b-1=2$ cards, we have $N_{\mathrm{card}}/(b-1)=5/2=2.5$. Only two full powers of $3$ can therefore be supported, corresponding to $3^0$ and $3^1$, while the cards required to include $3^2$ are incomplete. Taking the floor yields $\lfloor 5/2 \rfloor = 2$, and Eq.~(3) gives a maximum representable integer equal to $3^2-1=8$. Without the floor function, one would incorrectly assume access to the next power of the base.

\begin{figure}[H]
\centering
\includegraphics[width=1\textwidth]{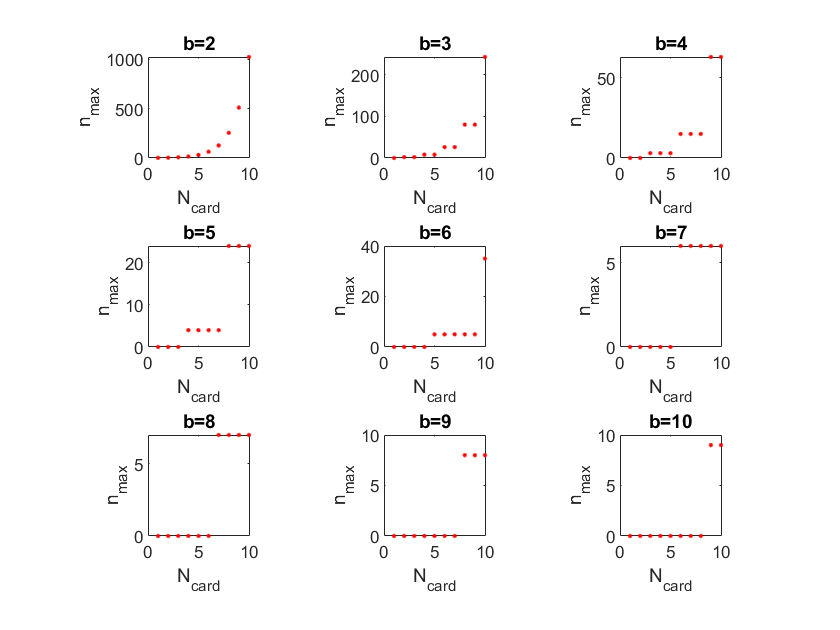}
\caption{Maximum representable integer as a function of the number of cards
$N_{card}$ for bases $b=2,\dots ,10$.}
\label{fig:figure5}
\end{figure}

\section*{Conclusion}

Starting from the classic street performance of Giuseppe Polone, we have developed a general framework for representing integers with collections of cards (or tables) indexed by the powers of an arbitrary base~$b$. This perspective leads naturally to an explicit formula for the number of cards required to identify any given integer and highlights, both visually and theoretically, the special efficiency of base~2.

Two principal results emerge. First, for every integer $n$, the binary representation requires the fewest cards to guarantee an unambiguous guess. Second, when the total number of cards is fixed, base~2 yields the widest range of integers that can be represented. Together these findings give precise quantitative meaning to the intuition that binary encoding is the most economical way to convey information means of information transmission for schemes based on yes/no queries

Beyond their recreational charm, these observations illustrate core ideas of mathematics and computer science. The game offers a concrete entry point to positional numeral systems, logarithmic growth, and the concept of information as counted in bits. It can enrich classroom discussions, inspire exploratory projects, or simply entertain audiences of all ages. Thus a simple street trick becomes a vivid demonstration of the universal power and elegance of binary representation.

\section*{Acknowledgements}
The authors are grateful to Dr.ssa~Serena Capanni, her son Lorenzo, and Dr.ssa~Francesca Contessini for a stimulating conversation about the story of the number-guessing game attributed to Giuseppe Polone.  
Their interest and insightful comments helped clarify the background of the tale and inspired the broader mathematical exploration presented in this paper.

\section*{Appendix}\label{sec:appendix}

\begin{theorem}\label{thm:1}
Let $N_{\mathrm{card}}(b)$ denote the minimum number of cards required to guess any integer between $1$ and $n$ when the guessing scheme is based on representations in base~$b$.
For every integer $n>0$ and every integer base $b>2$,
\[
N_{card}(2)=\bigl\lfloor \log_{2} n \bigr\rfloor + 1
\;\le\;
N_{card}(b)=(b-1)\bigl(\lfloor \log_{b} n \rfloor + 1\bigr).
\]
Moreover, equality holds only in the following cases:
\begin{enumerate}
   \item $b=3$ and $n=8$, where both sides equal $4$;
   \item $b=3$ and $n=2$, where both sides equal $2$.
\end{enumerate}
In all other cases the inequality is strict.
\end{theorem}

\begin{proof}
Although the claim feels intuitively reasonable, the verification is somewhat delicate because of the floor functions.

\medskip
\noindent\textbf{Step 1: Rewriting the inequality.}
Since
\[
(b-1)\bigl(\lfloor \log_{b} n \rfloor + 1 \bigr)
   =(b-1)\,\lfloor \log_{b} n \rfloor +(b-1),
\]
the desired inequality is equivalent to
\[
\lfloor \log_{2} n \rfloor
   \le (b-1)\,\lfloor \log_{b} n \rfloor + (b-2).
\]
Because $\log_{b} n=\log_{2} n/\log_{2} b$, we can write this as
\[
\lfloor \log_{2} n \rfloor
   \le (b-1)\,\bigl\lfloor \dfrac{\log_{2} n}{\log_{2} b} \bigr\rfloor + (b-2).
\]
Let
\[
m=\Bigl\lfloor \dfrac{\log_{2} n}{\log_{2} b}\Bigr\rfloor ,
\qquad
\log_{2} n = m\log_{2} b + r\log_{2} b
\quad\text{with } 0 \le r < 1.
\]

\medskip
\noindent\textbf{Step 2: Bases $b\ge 6$.}
If $n \ge b$ then $m \ge 1$. For real $x\ge6$, define
\[
f(x)=\dfrac{x-1}{2\log_{2} x}.
\]
A derivative calculation shows $f'(x)>0$ for $x>1$, hence $f$ is increasing and $f(6)>1$. Thus, for all $x \ge 6$,
\[
(x-1) \ge 2\log_{2} x .
\]
Therefore
\[
(x-1)m \ge 2m\log_{2} x
       = m\log_{2} x + m\log_{2} x
       \ge m\log_{2} x + r\log_{2} x
       = \log_{2} n .
\]
Since $x-2>0$,
\[
\lfloor \log_{2} n\rfloor<\log_{2} n < (x-1)m + (x-2),
\]
as required.

If instead $n<b$ we have $m=0$, so the right–hand side equals $b-2$. Using the fact that $\log_{2} b/(b-2)$ is decreasing for $b>1$ (by an argument entirely analogous to the derivative check above) and equals $1$ at $b=4$, we conclude that $\log_{2} b < b-2$ for every $b \ge 6$. Hence
\[
\lfloor \log_{2} n\rfloor<\log_{2} n < \log_{2} b < b-2,
\]
and the claim again holds.

\medskip
\noindent\textbf{Step 3: Intermediate bases $b=4$ and $b=5$.}
For $b=5$ we need
\[
3 + 4m > m\log_{2} 5 + \log_{2} 5
\quad\Longleftrightarrow\quad
m > \dfrac{\log_{2} 5 - 3}{4 - \log_{2} 5} \approx -0.40 .
\]
Since $m \ge 0$, this holds for all $n$.

For $b=4$ we obtain
\[
2 + 3m > m\log_{2} 4 + \log_{2} 4
\quad\Longleftrightarrow\quad
m > \dfrac{\log_{2} 4 - 2}{3 - \log_{2} 4} = 0 .
\]
Thus the inequality holds for every $n \ge 4$. When $1 \le n \le 3$, we have $\lfloor \log_{2} n \rfloor \in \{0,1\}$ and the right–hand side equals $(b-2)=2$, so the inequality is also satisfied.

\medskip
\noindent\textbf{Step 4: Base $b=3$.}
Here the condition becomes
\[
1 + 2m > m\log_{2} 3 + \log_{2} 3
\quad\Longleftrightarrow\quad
m > \dfrac{\log_{2} 3 - 1}{2 - \log_{2} 3} \approx 1.41 .
\]
Thus the inequality holds for $m \ge 2$, i.e.\ for $n \ge 9$. When $n=8$ we have equality:
\[
\lfloor \log_{2} 8 \rfloor = 3,
\qquad
2\lfloor \log_{3} 8 \rfloor + 1 = 2\cdot 1 + 1 = 3 .
\]
For $4 \le n \le 7$ we get
\[
\lfloor \log_{2} n \rfloor = 2,
\qquad
2\lfloor \log_{3} n \rfloor + 1 = 3,
\]
so the inequality is strict. Finally, for small $n$ we check directly:
\[
\begin{array}{lll}
n=3:& \lfloor \log_{2}3 \rfloor = 1, & 2\lfloor \log_{3}3 \rfloor + 1 = 3;\\[2pt]
n=2:& \lfloor \log_{2}2 \rfloor = 1, & 2\lfloor \log_{3}2 \rfloor + 1 = 1 \ \ (\text{equality});\\[2pt]
n=1:& \lfloor \log_{2}1 \rfloor = 0, & 2\lfloor \log_{3}1 \rfloor + 1 = 1.
\end{array}
\]

\medskip
Combining all cases, we conclude that
\[
\lfloor \log_{2} n \rfloor + 1
   \le (b-1)\bigl( \lfloor \log_{b} n \rfloor + 1 \bigr)
\]
for every integer $n>0$ and every integer base $b>2$, with equality occurring only when $(b,n)=(3,8)$ or $(3,2)$.
\end{proof}

\bigskip\bigskip\bigskip

\begin{theorem}\label{thm:2}
Let $N_{card}\in\mathbb{N}$ and $b>2$. Then
\[
2^{ N_{card}} - 1
   > b^{\lfloor N_{card}/(b-1)\rfloor} - 1 .
\]
\end{theorem}

\begin{proof}
Because
\[
b^{\,N_{card}/(b-1)} - 1 \;\ge\;
b^{\,\lfloor N_{card}/(b-1)\rfloor} - 1 ,
\]
it is enough to show $2^{N_{card}}-1 > b^{\,N_{card}/(b-1)}-1,
$
or equivalently $
2^{N_{card}} > b^{\,N_{card}/(b-1)} .
$
and obtain
$ b^{\,N_{card}/(b-1)}
   = 2^{\,(\log_{2} b)\,N_{card}/(b-1)} .
$\\
Dividing $2^{N_{card}}$ by this quantity gives
\[
\frac{2^{N_{card}}}{2^{(\log_{2} b)\,N_{card}/(b-1)}}
   = 2^{\,N_{card}\!\left(1-\frac{\log_{2} b}{\,b-1}\right)}
   = \bigl(2^{\,1-\frac{\log_{2} b}{\,b-1}}\bigr)^{N_{card}} .
\]
Thus $2^{N_{card}} > b^{N_{card}/(b-1)}$ holds precisely when $
2^{\,1-\frac{\log_{2} b}{\,b-1}} > 1$, that is when $
\frac{\log_{2} b}{\,b-1} < 1$.\\
The strict decrease of
\(
g(b)=\frac{\log_{2} b}{\,b-1}
\)
was established in the proof of Theorem~\ref{thm:1}. Since $g(2)=1$ and $g$ is decreasing, we have $g(b)<1$ for all $b>2$. Therefore $2^{\,1-\frac{\log_{2} b}{\,b-1}} > 1$,
and consequently
\[
2^{N_{card}}
  > 2^{(\log_{2} b)\,N_{card}/(b-1)}
  = b^{\,N_{card}/(b-1)}.
\]
It follows that
\[
2^{N_{card}} - 1
   > b^{\,N_{card}/(b-1)} - 1
   \ge b^{\,\lfloor N_{card}/(b-1)\rfloor} - 1 ,
\]
which completes the proof.
\end{proof}


\begin{thebibliography}{99}

\bibitem{Polone2025}
Giuseppe Polone e l’Università Stradale Matematica d’Italia.
\textit{La Gente di Napoli}.
Available at:
\url{https://www.lagentedinapoli.it/giuseppe-polone-e-luniversita-stradale-matematica-ditalia/}
Accessed: 2025-09-20.

\bibitem{polonephoto}
Giuseppe Polone, la calculadora humana de Nápoles: ``Después de Pitágoras y Einstein, vengo yo''.
\textit{The Clinic}.
Available at:
\url{https://www.theclinic.cl/2017/03/14/giuseppe-polone-la-calculadora-humana-napoles-despues-pitagoras-einstein-vengo/}
Accessed: 2025-09-20.

\bibitem{Gardner1958}
M. Gardner,
\textit{Mathematics, Magic and Mystery}.
Dover Publications, 1956, p.~95.

\bibitem{Yingprayoon2019}
Y. Yingprayoon,
``A Number Guessing Trick in Different Bases,''
in \textit{Proceedings of the 14th Asian Technology Conference in Mathematics (ATCM 2019)},
Bangkok, Thailand, 2019.
Contribution no.~7. Available at:
\url{https://atcm.mathandtech.org/EP2019/regular.html}
Accessed: 2026-01-16.

\bibitem{LongaBinaryCardTrick}
A. Long,
``The Binomary Card Trick.''
Available at:
\url{https://www.nku.edu/~longa/shows/binary/binomialcards.pdf}
Accessed: 2026-01-16.

\bibitem{ChalkdustBinaryTrick2016}
A. Doak,
``Binary magic card trick,''
\textit{Chalkdust Magazine}, 2016.
Available at:
\url{https://chalkdustmagazine.com/advent-calendar/09-december/}
Accessed: 2026-01-16.

\bibitem{Sezin2009}
F. Sezin,
``Number Guessing,''
\textit{The Australian Mathematics Teacher},
vol.~65, no.~3, pp.~10--16, 2009.

\end{thebibliography}
\end{document}